\documentclass[11pt]{amsart}
\usepackage{amsmath, amssymb, amsthm}
\usepackage{enumitem}                                 
\usepackage[margin=1.2in]{geometry}

\AtEndDocument{\bigskip{\footnotesize%
\textsc{Department of Mathematics, the Pennsylvania State University, University
Park, PA 16802, USA} \par
  \textit{E-mail address} \texttt{ aka5983@psu.edu
}}

\bigskip{\footnotesize%
\textsc{School of Mathematics, University of Birmingham, Birmingham, B15 2TT, UK} \par
\textit{E-mail address} \texttt{  simonbaker412@gmail.com
}}

\bigskip{\footnotesize%
\textsc{Department of Mathematics, the University of British Columbia, 1984 Mathematics Road Vancouver, BC, V6T 1Z2, Canada} \par
\textit{E-mail address} \texttt{ pshmerkin@math.ubc.ca
}
}}

\DeclareMathOperator{\supp}{supp}

\newtheorem{theorem}{Theorem}[section]

\newtheorem{Counter-example}[theorem]{Counter example}
\newtheorem{Claim}[theorem]{Claim}

\newtheorem{Lemma}[theorem]{Lemma}

\newtheorem*{theorem*}{Theorem}

\newcommand{\ignore}[1]{}

\usepackage{xcolor}
\usepackage[pagebackref]{hyperref}
\hypersetup{
   colorlinks,
    linkcolor={red!60!black},
    citecolor={blue!60!black},
    urlcolor={blue!90!black}
}

\title{On normal numbers and self-similar measures}

\author{Amir Algom, Simon Baker, and Pablo Shmerkin}
\date{}

\thanks{P.S. was supported by an NSERC discovery grant}

\begin{document}
\maketitle
\begin{abstract}
Let $\lbrace f_i(x)=s_i \cdot x+t_i \rbrace$ be a self-similar IFS on $\mathbb{R}$ and let $\beta >1$ be a Pisot number. We prove that if $\frac{\log |s_i|}{\log \beta}\notin \mathbb{Q}$ for some $i$   then for every $C^1$ diffeomorphism $g$ and every non-atomic self similar measure $\mu$, the measure $g\mu$ is supported on numbers that are normal in base $\beta$.
\end{abstract}

\section{Introduction}

\subsection{Background and  main result} \label{Section statement}
Let $b$ be an integer greater or equal to $2$. Let $T_b$ be the times-$b$ map,
$$T_b (x) = b\cdot x \mod 1, \, x\in \mathbb{R}.$$
A number $x\in \mathbb{R}$ is called $b$-normal, or normal in base $b$, if its orbit $\lbrace T_b ^k (x) \rbrace_{k\in \mathbb{N}}$  equidistributes for the Lebesgue measure on $[0,1]$. In 1909 Borel proved  that Lebesgue almost every $x$ is \textit{absolutely normal}, that is, normal in all bases. In the absence of obvious obstructions, it is believed that this phenomenon should continue to hold true for typical elements of well structured sets with respect to appropriate measures. The purpose of this paper is to make a contribution in this direction for a well studied class of fractal measures.

Before stating our main result, we recall that one may define digit expansions of numbers in non-integer bases as well: Following R\'{e}nyi \cite{Renyi1967rep}, for $\beta>1$ we define the $\beta$-expansion of $x\in (0,1)$ to be the lexicographically largest sequence $x_n \in \lbrace 0,...,[\beta]\rbrace$ such that $x=\sum_{n=1} ^\infty x_n \beta^{-n}$. This sequence is obtained from the orbit of $x$ under the map $T_\beta(x)= \beta\cdot  x \mod 1$ similarly to the integer case. It is known that $T_\beta$ admits a unique absolutely continuous invariant measure, commonly referred to as the Parry measure \cite{Parry1960beta}. When $\beta$ is an integer then the Parry measure is just the Lebesgue measure restricted to $[0,1)$. Thus, in analogy with the integer case, we say that $x$ is $\beta$-normal if it equidistributes under $T_\beta$ for the Parry measure. Our main result will allow us to conclude that for certain fractal measures, a typical element will be $\beta$-normal, where $\beta>1$ may be a non-integer.

In this paper we will work with self-similar measures on $\mathbb{R}$. These are defined as follows:  Let $\Phi= \lbrace f_1,...,f_n \rbrace$ be a finite set of real non-singular contracting similarity maps of a compact interval $J\subset \mathbb{R}$. That is, for every $i$ we can write $f_i(x)=s_i \cdot x+t_i$ where $s_i \in (-1,1)\setminus \lbrace 0 \rbrace$ and $t_i \in \mathbb{R}$, and $f_i(J)\subseteq J$. We will refer to $\Phi$ as a \textit{self similar IFS} (Iterated Function System).  It is well known that there exists a unique compact set $\emptyset \neq K=K_\Phi\subseteq J$ such that
$$ K = \bigcup_{i=1} ^n f_i (K).$$
The set $K$ is called a \textit{self-similar set}, and the \textit{attractor} of the IFS $\Phi$.  We always assume that there exist $i\neq j$ such that the fixed point of $f_i$ is not equal to the fixed point of $f_j$. This ensures that $K$ is infinite.

 Next, let $\textbf{p}=(p_1,...,p_n)$ be a strictly positive probability vector, that is, $p_i >0$ for all $i$ and $\sum_i p_i =1$. It is well known that there exists a unique Borel probability  measure $\mu$ such that
$$\mu = \sum_{i=1} ^n p_i\cdot  f_i\mu,\quad \text{ where } f_i \mu \text{ is the push-forward of } \mu \text{ via } f_i.$$
The measure $\mu$ is called a \textit{self-similar measure}, and is supported on $K$.  Our assumptions that $K$ is  infinite and  that $p_i>0$ for every $i$ are known to imply that $\mu$ is non-atomic. In particular, all self-similar measures in this paper are  non-atomic.

Recall that $\beta>1$ is called a \textit{Pisot} number if it is an algebraic integer whose algebraic conjugates are all of modulus strictly less than $1$. Note that every integer larger than $1$ is a Pisot number. Also, we will write $a\sim b$ if the real numbers $a, b \neq 0$ satisfy that $\frac{\log |a|}{\log |b|} \in \mathbb{Q}$. Otherwise, we write $a \not \sim b$, in which case $a,b$ are said to be \textit{multiplicatively independent}. Finally, given a measure $\mu$ on $\mathbb{R}$ and $\beta>1,$ we say that $\mu$ is pointwise $\beta$-normal if $\mu$ almost every $x$ is $\beta$-normal. We are now ready to state the main result of this paper:
\begin{theorem} \label{Main Theorem}
Let $\beta >1$ be a Pisot number, and let $\mu$ be a non-atomic self-similar measure with respect to the self-similar IFS $\lbrace f_i(x)=s_i \cdot x+t_i \rbrace$ and a strictly positive probability vector. If there is some $i$ such that $s_i \not \sim \beta$ then $\mu$ is pointwise $\beta$-normal. Furthermore, $g\mu$ is pointwise $\beta$-normal for all $g\in \text{diff}^1 (\mathbb{R})$.
\end{theorem}
We emphasize that  no separation condition is imposed on the underlying IFS.  Theorem \ref{Main Theorem} is sharp in the sense that there are many IFSs such that $s_i=n$ for all $i$ and some integer $n$, and no element in the attractor is $n$-normal  - for example, no element in the middle-thirds Cantor set is $3$-normal. Nonetheless, at least for integer $\beta$, we will soon recall some recent results showing that even if all the contractions $s_i \sim \beta$ it is still possible that every self-similar measure is pointwise $\beta$-normal. We also note that there are other ways self-similar measures can lead to equidistribution - see e.g. \cite{bakwer2021equi}

Theorem \ref{Main Theorem}  extends a long line of research on pointwise normality for dynamically defined measures, and provides a unified proof for many recent results regarding  self similar measures.  The first instances of a special case of Theorem \ref{Main Theorem} were obtained by Cassels \cite{Cassels1960normal} and Schmidt \cite{Schmidt1960normal} around 1960. They considered the Cantor-Lebesgue measure on the middle thirds Cantor set, proving that almost every point is normal to base $n$ as long as $n \not \sim 3$.  This was later generalized by Feldman and Smorodinsky \cite{Feldman1992normal} to all non-degenerate Cantor-Lebesgue measures   with respect to any base $b$. We remark that the work of Host \cite{Host1995normal} and the subsequent works of  Meiri \cite{Meiri1998host} and Lindenstrauss \cite{Elon2001host} about pointwise normality for $T_n$-invariant measures are also related to this, though in the somewhat different context of Furstenberg's $\times 2 ,\times 3$ Conjecture.

In 2015, Hochman and Shmerkin \cite[Theorem 1.4]{hochmanshmerkin2015} proved Theorem \ref{Main Theorem} under the additional assumptions that $s_i>0$ for all $i$, and the intervals $f_i(J)$, $i\in J$ are disjoint except potentially at their endpoints (a condition stronger than the open set condition). In fact, \cite[Theorem 1.2]{hochmanshmerkin2015} provides a general criterion for a uniformly scaling measure (that is, a measure such that at  almost every point its scenery flow equidistributes for the same  distribution) to be pointwise $\beta$-normal. We refer the reader to Section \S\ref{Section outline} for the relevant definitions about the scenery flow. The main reason they imposed this separation condition was to ensure that the corresponding self-similar measures are indeed uniformly scaling. This fact had previously been established by Hochman \cite[Section 4.3]{hochman2010dynamics}. Recently, Py{\"o}r{\"a}l{\"a} \cite{Aleksi2021flow} proved that it suffices to assume that the IFS satisfies the weak separation condition \cite{Lau1999weak, Zerner1996weak} in order for self similar measures to be uniformly scaling. Utilizing this result and the methods from \cite{hochmanshmerkin2015}, Py{\"o}r{\"a}l{\"a} obtained a version of Theorem \ref{Main Theorem} under the additional assumption that the IFS satisfies the weak separation condition \cite[Corollary 5.4]{Aleksi2021flow}. We note that there are many self-similar sets which satisfy neither the open set nor the weak separation condition; indeed, this is often the generic behaviour in parametrized families of self-similar sets: see \cite{PSS00}.

 Very recently, Algom, Rodriguez Hertz, and Wang, proved that if the Fourier transform of a self similar measure decays to $0$ at $\infty$ then almost every point is absolutely normal \cite[Theorem 1.4]{algom2020decay}. Combining this with recent progress on the Fourier decay problem for self similar measures leads to many special cases of Theorem \ref{Main Theorem}: For example, by a result of Li and Sahlsten \cite{li2019trigonometric}, if there are $i,j$ such that $s_i \not \sim s_j$ then all self similar measures have this property, and consequently are supported on absolutely normal numbers. Furthermore, via \cite[Theorem 1.4]{algom2020decay} and the recent work of Br\'{e}mont \cite{bremont2019rajchman} one obtains many instances of self similar measures that are supported on $n$-normal numbers even though the underlying IFS satisfies $s_i \sim n$ for all $i$. Examples of this form were also recently obtained by Dayan, Ganguly, and Weiss \cite{dayan2020random}.

The results of \cite{algom2020decay} are related to a classical theorem of Davenport-Erd\H{o}s-LeVeque \cite{Davenport1964Erdos}. This result states that if the Fourier transform of a Borel measure decays to zero sufficiently quickly, then a typical number with respect to this measure will be normal (for integer bases). Recall that by \cite[Theorem 1.4]{algom2020decay}, this property holds for self-similar measures with a decaying Fourier transform, regardless of the rate of decay. In \cite{Shmerkin2018mos} it was shown that for any homogeneous IFS with non-atomic self-similar measure $\mu$, if $g$ is a $C^2$ function satisfying $g''>0$ then the Fourier transform of $g\mu$ decays to zero sufficiently quickly so that the Davenport-Erd\H{o}s-LeVeque theorem applies. For dynamically defined measures there are many recent results that establish a sufficiently fast rate of decay for this result to apply: \cite{algom2021decay, Sahl2016Jor, sahlsten2020fourier, li2019trigonometric, solomyak2019fourier, varju2020fourier, rapaport2021rajchman}    to name just a few examples (see e.g. \cite{algom2020decay} for many more references).

Let us compare Theorem \ref{Main Theorem} with the results of \cite{algom2020decay, hochmanshmerkin2015, Aleksi2021flow}: First, we fully relax the separation conditions on the IFS from \cite{hochmanshmerkin2015, Aleksi2021flow}. Second, while no separation conditions are needed for \cite[Theorem 1.4]{algom2020decay}, it  does not cover two important cases where Theorem \ref{Main Theorem} does apply: The  general case when $s_i \sim s_j$ for all $i,j$, and pointwise normality for non-integer Pisot numbers. Third, our work gives a simple and unified approach to this problem that is relatively self-contained. In particular, we do not need to invoke any results on the Fourier transform of self-similar measures as in \cite{algom2020decay}. On the other hand, we note that the methods and results of both \cite{hochmanshmerkin2015} and \cite{algom2020decay} apply for more general self-conformal IFS's (i.e. they allow for non-affine smooth maps in the IFS), and that most of the results of \cite{hochmanshmerkin2015} work for a broader class of measures on the fractal.

We conclude this section with a brief informal discussion of our method. It consists of three steps: Let $\mu$ be a self similar measure as in Theorem \ref{Main Theorem}, and let $\beta$ be a Pisot number such that $\beta \not \sim s_i$ for some $i$. We want to show that $\mu$ is pointwise $\beta$-normal. The first and most critical step  is to express $\mu$ as an integral over a certain family of random measures. This is based upon a technique that first appeared in \cite{Galcier2016Lq}, and was subsequently applied in \cite{Antti2018orponen} and \cite{Shmerkin2018Solomyak}, to study the dimension of $\mu$. What was
important for these authors was that these random measures could be expressed as an infinite
convolution. The crucial new feature in our construction is to ensure that these random measures typically satisfy a certain dynamical self-similarity relation with \textit{strong separation}. We are also able to preserve the arithmetic condition $\beta \not \sim s_i$, in some sense, into these random measures.

In the second step of the proof, we show that the scenery flow of typical random measures satisfying this dynamical self-similarity relation with strong separation equidistributes for the same non-trivial ergodic distribution, which we construct explicitly. 

The third and final step of the proof is to use spectral analysis of this distribution together with independence from $\beta$ to conclude, via  \cite[Theorem 1.2]{hochmanshmerkin2015}, that almost every random measure is pointwise $\beta$-normal.  Since in the first step we disintegrated $\mu$ according to these measures, Theorem \ref{Main Theorem} follows. The assertion made in Theorem \ref{Main Theorem} about $C^1$ pushforwards of $\mu$ also follows along these lines, by noting that $C^1$ images of typical random measures will still be pointwise $\beta$-normal by \cite[Theorem 1.2]{hochmanshmerkin2015}.

In the next section, after recalling some definitions, we make this sketch precise and formally state the main steps in the proof.

\subsection{Statement of the main steps in the proof} \label{Section outline}
We begin by recalling the notion of a model (as in e.g. \cite{Shmerkin2018Solomyak}):
Let $I$ be a  finite set of iterated function systems of similarities $\Phi^{(i)} = (f_1 ^{(i)},..., f_{k_i} ^{(i)})$, $i\in I$. We assume  that each IFS is homogeneous and uniformly contracting. That is, for every $i\in I$ there is some $r_i\in (-1,0)\cup (0,1)$ such that for every $1\leq j \leq k_i$ we have
$$f^{(i)} _j (x) = r_i \cdot x + t^{(i)} _j, \, \text{ where } t^{(i)} _j\in \mathbb{R}.$$
We allow $k_i$ to equal $1$ for some $i$, i.e. to have degenerate iterated function systems.

Let $\Omega = I^\mathbb{N}$. Given a sequence $\omega = (\omega_n)_{n\in \mathbb{N}} \in  \Omega $ we define the space of words of length $n$ (possibly with $n = \infty$) with respect to $\omega$ via
$$X_n ^{(\omega)} = \prod_{j=1} ^n \lbrace 1,...,k_{\omega_j} \rbrace.$$
Next, for every $\omega$ we define subsets of $\mathbb{R}$ via
$$Y^{(\omega)} = \left\lbrace \sum_{n=1} ^\infty \left( \prod_{j=1} ^{n-1} r_{\omega_j} \right) \cdot t^{(\omega_n)} _{u_n} :\, u \in X_\infty ^{(\omega)} \right\rbrace.$$
Thus, for every  $\omega$ we have a surjective coding map $\Pi_\omega : X^{(\omega)} _\infty \rightarrow Y^{(\omega)}$ defined by
\begin{equation} \label{eq:coding-map}
\Pi_\omega (u)= \sum_{n=1} ^\infty \left( \prod_{j=1} ^{n-1} r_{\omega_j} \right) \cdot t^{(\omega_n)} _{u_n},\quad u\in  X^{(\omega)} _\infty.
\end{equation}

Let $\sigma$ denote the left-shift on $\Omega$ (or other shift spaces), and note that for every $\omega\in\Omega$ the following \emph{dynamical self-similarity} relation holds:
\begin{equation} \label{Eq union}
Y ^{(\omega)} =  \bigcup_{u\in X_1 ^{(\omega)}} f_{u} ^{(\omega_1)} \left( Y^{\sigma(\omega)} \right).
\end{equation}
We say that the model satisfies the \textit{Strong Separation Condition} (SSC) if the union in  \eqref{Eq union} is disjoint for all $\omega\in\Omega$. Here we adopt the convention that the union of a single set is considered to be disjoint.

Next, for each $i\in I$, let $\mathbf{p}_i = (p_1 ^{(i)},...,p_{k_i} ^{(i)})$ be a probability vector with strictly positive entries. On each $X_\infty ^{(\omega)}$  we can then define the product measure $\bar{\eta} ^{(\omega)} := \prod_{n=1} ^\infty \mathbf{p}_{\omega_n}$. The projection of $\bar{\eta} ^{(\omega)}$ via the coding map is a Borel probability measure $\eta^{(\omega)}$   supported on $Y^{(\omega)}$.  For every $\omega$ the measure $\eta^{(\omega)}$  also satisfies a dynamical self similarity relation, namely,
\begin{equation} \label{eq:self-similarity}
\eta^{(\omega)} = \sum_{u\in X_1 ^{(\omega)}} p_u ^{(\omega_1)} f_u ^{(\omega)} \eta^{(\sigma(\omega))}.
\end{equation}

Finally, let $\mathbb{P}$ be a $\sigma$-invariant measure on $\Omega$.  We refer  to the triple $\Sigma = ( \Phi^{(i)} _{i\in I}, (\mathbf{p}_i)_{i\in I}, \mathbb{P})$ as the model under consideration. We say that the model is \textit{non-trivial}  if $\eta^{(\omega)}$ is non-atomic for $\mathbb{P}$-a.e. $\omega$. We say that the model is \textit{ergodic} if the selection measure $\mathbb{P}$ is an ergodic measure on $\Omega$. We say that the model is \textit{Bernoulli} if the selection measure $\mathbb{P}$ is a Bernoulli measure on $\Omega$.

As indicated earlier, the proof of Theorem \ref{Main Theorem} consists of three main steps. The first step is to prove that every self similar measure admits a disintegration over the typical measures of a model. Furthermore, this model is Bernoulli, and the arithmetic properties of the original IFS are preserved in an appropriate sense:
\begin{theorem} \label{Theorem self similar}
Let $\mu$ be a non-atomic self similar measure with respect to an IFS $\lbrace \varphi_i(x)=s_i \cdot x+t_i \rbrace$, and let $\beta>1$. Then there exists a non-trivial Bernoulli model $\Sigma=\Sigma(\mu)$ with strong separation such that:
\begin{enumerate}
\item $\mu = \int \eta^{(\omega)} d \mathbb{P}(\omega)$.

\item If  $s_j \not \sim \beta$ for some $j$ then there exist $i\in I$ such that $r_i \not \sim \beta$.
\end{enumerate}
\end{theorem}
This theorem is proved in Section \S\ref{Section proof of Theorem self similar}.

The second step in the proof of Theorem \ref{Main Theorem} is a general result about random model measures being uniformly scaling.  Before stating this result, we recall the definition of the scaling scenery of a measure, and related notions.  We  follow the notations as in \cite{hochmanshmerkin2015}.

Let $\mathcal{P}(X)$ denote the family of Borel probability measures on a metrizable space $X$, and define
\begin{equation} \label{M square}
\mathcal{M}^{\square} = \lbrace \mu \in \mathcal{P}([-1,1]):\, 0 \in \text{supp} (\mu) \rbrace.
\end{equation}
For $\mu\in \mathcal{P}(\mathbb{R})$ and $x\in \text{supp} (\mu)$, we define the translated and renormalized measure $\mu_{x}\in \mathcal{M}^{\square}$ by
	\begin{equation*}
	\mu_x(E) = c \cdot  \mu ( E+x),\quad \text{for }E\subset[-1,1] \text{ a Borel set.}
	\end{equation*}
Here $c=\mu([x-1,x+1])^{-1}$ is a normalizing constant. Note that in the definition of $\mu_x$, the measure $\mu$ does not need to be supported on $[-1,1]$ (but in any case $\mu_x\in\mathcal{M}^{\square}$).

For $\mu \in \mathcal{M}^{\square}$, we define the scaled measure $S_t \mu \in \mathcal{M}^{\square}$ by
	\begin{equation*}
	S_t \mu (E) = c' \cdot \mu (e^{-t}E),\quad \text{for }E\subset[-1,1] \text{ a Borel set},
	\end{equation*}
	where $c'=\mu([-e^{-t},e^{-t}])^{-1}$ is again a normalizing constant.

The \emph{scaling flow} is the Borel $\mathbb{R}^+$ flow $S=(S_t)_{t\geq0}$ acting on $\mathcal{M}^{\square} $. The \emph{scenery} of $\mu$ at $x\in \text{supp}(\mu)$ is the orbit of $\mu_x$ under $S$, that is, the one parameter family of measures $\mu_{x,t}:= S_t(\mu_x)$ for $t\geq0$. Thus, the scenery of the measure at some point $x$ is what one sees as one ``zooms into'' the measure with focal point $x$.

Notice that $\mathcal{P}(\mathcal{M}^{\square} ) \subseteq \mathcal{P}(\mathcal{P}([-1,1]))$. As is standard in this context, we shall refer to elements of $\mathcal{P}(\mathcal{P}([-1,1]))$ as distributions (and denote them by capital letters such as $Q$), and to elements of $ \mathcal{P}(\mathbb{R})$ as measures (and denote them by Greek letters such as $\mu$). A measure $\mu \in \mathcal{P}(\mathbb{R})$ \textit{generates a distribution} $P\in \mathcal{P}(\mathcal{P}([-1,1]))$ at $x\in \text{supp} (\mu)$ if the scenery at $x$ equidistributes for $P$ in $\mathcal{P}(\mathcal{P}([-1,1]))$, i.e. if
\begin{equation*}
\lim_{T\rightarrow \infty} \frac{1}{T} \int_0 ^T f(\mu_{x,t}) dt = \int f(\nu) dP(\nu),\quad \text{ for all } f\in C(\mathcal{P}([-1,1])).
\end{equation*}
(Any limits involving measures are understood to be in the weak topology.) We say that \textit{$\mu$ generates $P$}, and call $\mu$ a \textit{uniformly scaling measure},  if it generates $P$ at $\mu$ almost every $x$. If $\mu$ generates $P$, then $P$ is supported on $\mathcal{M}^{\square}$ and is $S$-invariant \cite[Theorem 1.7]{hochman2010dynamics}. Moreover, $P$ satisfies a sort of translation invariance known as the \emph{quasi-Palm} condition (see \cite{hochman2010dynamics}); we will not use this fact directly but it plays a key role in the proof of \cite[Theorem 1.2]{hochmanshmerkin2015}) which is one of the main tools in the proof of Theorem \ref{Main Theorem} (and is recalled as Theorem \ref{Thm: H-S} below).

We say that $P$ is \textit{trivial} if it is the distribution supported on the atom $\delta_{0} \in \mathcal{M}^{\square}$ - a fixed point of $S$. We can now state the second  main step towards the proof of Theorem \ref{Main Theorem}:
\begin{theorem} \label{Theorem model}
Let $\Sigma$ be a non-trivial Bernoulli model with strong separation. Then there exists a non-trivial  $S$-ergodic distribution $Q$ such that for $\mathbb{P}$-a.e. $\omega$ the measure $\eta^{(\omega)}$ generates $Q$.
\end{theorem}
In fact, we will derive an explicit expression for $Q$ as a factor of a suspension flow over a Markov measure. This theorem is proved in Section \S\ref{Section proof of Theorem model}.

The third and final step is to deduce Theorem \ref{Main Theorem} from Theorem \ref{Theorem self similar} and Theorem \ref{Theorem model}. To do this, we  study the pure point spectrum of the corresponding distribution $Q$, showing that it does not contain a non-zero integer multiple of $\frac{1}{\log \beta}$. Once this is established, we apply Theorem \ref{Thm: H-S}  to obtain the desired  pointwise normality result. See Section \S\ref{Section proof of main result} for more details.

\section{Proof of Theorem \ref{Theorem self similar}} \label{Section proof of Theorem self similar}
Let $\Phi = \lbrace \varphi_i(x)=s_i \cdot x+t_i \rbrace_{i=1} ^n$ be a self similar IFS. Fixing weights $\mathbf{p}$, we begin by constructing a model $\Sigma=\Sigma(\Phi,\mathbf{p})$. We then show that it meets the requirements of Theorem \ref{Theorem self similar}.
\subsection{Construction of the model} \label{Section construction of model}
Write $\mathcal{A}:=\lbrace 1,...,n \rbrace$ and let $\mathbf{p}$ be a strictly positive probability vector on $\mathcal{A}$. For every $M\in \mathbb{N}$ and $I=(i_1,...,i_m)\in \mathcal{A}^M$ write $\varphi_I =  \varphi_{i_1}\circ \cdots \circ \varphi_{i_M}$, and  let
$$ \Phi^M := \lbrace \varphi_I:\, I \in \mathcal{A}^M \rbrace.$$
Recall that we are assuming that the attractor $K_\Phi$ is infinite. It is not hard to check (see \cite[Proof of Lemma 4.2]{shmerkin2015projections}) that  there exist some $M\in\mathbb{N}$ and $I,J\in \mathcal{A}^M$ such that
$$\text{Conv} \left( \varphi_I (K_{\Phi}) \right) \bigcap \text{Conv} \left( \varphi_J (K_{\Phi}) \right) = \emptyset \, \text{ and } \varphi_I ' = \varphi_J ',$$
where $\text{Conv}(X)$ denotes the convex hull of a set $X$. In addition, notice that $\mu$, the self similar measure corresponding to $\Phi$ and $\mathbf{p}$, is also a self similar measure with respect to the IFS $\Phi^m$, with weights
$$\mathbf{p}^M = \lbrace  p_{i_1}\cdot \cdot \cdot p_{i_M}:\, i_j \in \mathcal{A} \text{ for all } j\rbrace.$$
The contraction factors of $\Phi^m$ include $s_i^m$, so the hypothesis that some $s_i\not\sim \beta$ is preserved by this iteration. This shows that we may assume, without loss of generality, that there exists $i,j\in \mathcal{A}$ such that
\begin{equation} \label{Eq implies SSC}
\text{Conv} \left( \varphi_i (K_{\Phi}) \right) \bigcap \text{Conv} \left( \varphi_j (K_{\Phi}) \right) = \emptyset \, \text{ and } \varphi_i ' = \varphi_j '.
\end{equation}
After relabeling, we may also assume $i=1,j=2$. We will use \eqref{Eq implies SSC} to show that the model satisfies the SSC.

We now define $I := \lbrace j:\, j=0 \text{ or } j\in \mathcal{A}\setminus\{1,2\} \rbrace$.
 For every index in $I$ we can associate an IFS as follows:
\begin{itemize}
\item For $j=0$ we associate the IFS $\lbrace \varphi_1, \varphi_2\rbrace$. 
\item For every $j\in I\setminus\{0\},$ we associate the degenerate IFS $\lbrace \varphi_j\rbrace$.
\end{itemize}
Notice that the construction above gives us a dichotomy: Each IFS in $I$ is either degenerate, or is homogeneous with strong separation in the sense of \eqref{Eq implies SSC}.

Recall that $\mathbf{p}$ was our fixed probability vector on $\mathcal{A}$. We now define a probability vector $\mathbf{q}$ on $I$ as follows:
\begin{itemize}
\item The mass $\mathbf{q}$ gives to $0$ is $p_1 + p_2$, that is, $q_0 = p_1 + p_2$.

\item For every $j \in I\setminus\{0\}$, the singleton $j$ gets mass $q_j= p_j$.
\end{itemize}
Recalling our notation from Section \S\ref{Section construction of model}, let $\mathbb{P}$ be the corresponding Bernoulli measure $\mathbf{q}^\mathbb{N}$ on $I^\mathbb{N}:=\Omega$. This $\mathbb{P}$ will be our (Bernoulli) selection measure.

Next, for $0\in I$ we define the probability vector
$$\tilde{\mathbf{p}}_{0} = \left(\frac{p_1}{p_1+p_2},\, \frac{p_2}{p_1+p_2} \right) = \left(\frac{p_1}{q_0},\, \frac{p_2}{q_0} \right). $$
For every $j\in I\setminus\{0\}$  we define the (degenerate) probability vector
$$\tilde{\mathbf{p}}_{j} = (1).$$
Recall from Section \S\ref{Section outline} that with these probability vectors, on each $X_\infty ^{(\omega)}$ we can then define the product measure $\bar{\eta} ^{(\omega)}$, and the projection of $\bar{\eta} ^{(\omega)}$ via the coding map is a Borel probability measure $\eta^{(\omega)}$ supported on $Y^{(\omega)}$.

\subsection{Proof of the required properties}
In this section we show that the model $\Sigma$ constructed in Section \S\ref{Section construction of model} satisfies the conclusion of Theorem \ref{Theorem self similar}. First,  the model $\Sigma$  is  Bernoulli by definition. Secondly,  the union in \eqref{Eq union} is disjoint for all $\omega$: if $\omega_1\neq 0$ there is nothing to do; otherwise, this follows from \eqref{Eq implies SSC} and the fact that all sets $Y^{(\omega')}$ are contained in the original attractor $K$. Since among the $|r_i|$ we find integer powers of all the original $|s_j|$, if there is $j$ such that $s_j \not \sim \beta$, then there is also some $i\in I$ such that $r_i \not \sim \beta$.

We proceed to prove part (1) of Theorem \ref{Theorem self similar}: Let  $\mu$ be the self similar measure that corresponds to the weights $\mathbf{p}$.   We will show that
$$\int \eta^{(\omega)} d\mathbb{P}(\omega) = \sum_{i\in \mathcal{A}} p_i \varphi_i  \int \eta^{(\omega)} d\mathbb{P}(\omega).$$
Since $\mu$ is the unique Borel probability measure satisfying this self-similarity relation, it will then follow that
$$ \mu =  \int \eta^{(\omega)} d\mathbb{P}(\omega).$$
We have
\begin{eqnarray*}
\int \eta^{(\omega)} d\mathbb{P}(\omega) &=& \int \left( \sum_{u\in X_1 ^{(\omega)}} \tilde{p}_u ^{(\omega_1)} f_u ^{(\omega)}  \eta^{(\sigma (\omega))} \right) d\mathbb{P}(\omega) \\
&=& \sum_{i\in I} \int_{[i]} \left( \sum_{u\in X_1 ^{(\omega)}} \tilde{p}_u ^{(\omega_1)} f_u ^{(\omega)}  \eta^{(\sigma (\omega))} \right) d\mathbb{P}(\omega) \\
&=& \sum_{j\in I, j\neq 0} p_j \int \left( \varphi_j  \eta^{(\sigma (j*\omega))} \right) d\mathbb{P}(\omega)  \\
&+& (p_1 + p_2)\cdot \int \left( \frac{p_1 }{p_1 + p_2} \varphi_1 \eta^{(\sigma (1*\omega))} + \frac{p_2 }{p_1 + p_2} \varphi_2 \eta^{(\sigma (2*\omega))} \right) d\mathbb{P}(\omega)\\
&=& \sum_{i\in \mathcal{A}} p_i \varphi_i \int \eta^{(\omega)} d\mathbb{P}(\omega).
\end{eqnarray*}
Which is what we claimed.

So far we have established all the claims in Theorem \ref{Theorem self similar}, except for the non-degeneracy of the model: We claim that for $\mathbb{P}$-a.e. $\omega$, the measure $\eta^{(\omega)}$ is non-atomic. Indeed, by the SSC for any $\omega$ the coding map $\Pi_\omega$ is injective. So by our choice of weights it follows that for every $n\in \mathbb{N}$ and $I_n \in X_n ^{(\omega)}$, we have
\begin{equation} \label{Eq non-deg}
\eta^{(\omega)} \left( \Pi_\omega \left(  I_n \right) \right) =  \bar{\eta} ^{(\omega)} \left( \lbrace I_n  \rbrace \right) \leq \left( \max \left\lbrace \frac{p_1 }{p_1 + p_2},\, \frac{p_2 }{p_1 + p_2} \right\rbrace \right)^{ \left|\lbrace 1 \leq k \leq n:\, \omega_k = 0 \rbrace \right|}.
\end{equation}
Since $\mathbb{P}$ is a Bernoulli measure on $I^\mathbb{N}$ with strictly positive weights, $\mathbb{P}$-a.s. the digit $0$ occurs in $\omega$ infinitely times. This shows that the right hand side of \eqref{Eq non-deg} tends to $0$ as $n\rightarrow \infty$ for $\mathbb{P}$-a.e. $\omega$. Since any point in $X_\infty^{(\omega)}$ is covered by sets $\Pi_\omega \left(  I_n \right)$ for all $n$, the measures $\eta^{(\omega)}$ are $\mathbb{P}$-a.s. non-atomic, as claimed. This concludes the proof of Theorem \ref{Theorem self similar}.

\section{Proof of Theorem \ref{Theorem model}} \label{Section proof of Theorem model}
We split out proof of Theorem \ref{Theorem model} into two cases. We first detail the case when each IFS in our model is orientation preserving, i.e. $r_{i}\in(0,1)$ for all $i\in I$; this case avoids certain technicalities. We then consider the case when we have at least one $i$ satisfying $r_i\in(-1,0)$.

\subsection{The orientation preserving case}
Assume the conditions of Theorem \ref{Theorem model} hold true. In particular, recall that we are assuming the SSC \eqref{Eq union}, and that the selection measure $\mathbb{P}$ is a non-degenerate Bernoulli measure.  Without loss of generality, we may assume that the sets $(f_u^{(\omega_1)}(Y^{(\sigma\omega)}))_{u\in X_1^{(\omega)}}$ are at distance $>2$ from each other for all $\omega\in\Omega$. See the discussion in \S\ref{subsec:removing-dist} on how to treat the general case.

First, we define the random pair $(\omega, u)$, where $\omega$  is drawn according to $\mathbb{P}$  and $u$ is drawn according to $\overline{\eta}^{(\omega)}.$  Let $\mathbb{P}'$ denote this distribution. In fact, $\mathbb{P}'$ is the Bernoulli measure on the symbols $\{ (i,u): i\in I, 1\le u\le k_i\}$ with weights $q_{i,u}=\mathbb{P}([i])\cdot p_u^{(i)}$. We let $M$ be the shift map on $\Omega'=\supp(\mathbb{P}')$. In particular,  $\mathbb{P}'$ is $M$-invariant and ergodic.  
Write $\eta_{\omega,u}=\eta^{(\omega)}_{\Pi_{\omega}(u)}$ for simplicity; recall that this is the translation of $\eta^{(\omega)}$ that centers it at $\Pi_{\omega}(u)$.

\begin{Claim} \label{Claim-M'} We have:
$S_{-\log(r_{\omega_1})}\eta_{\omega,u} = \eta_{M(\omega,u)}$.
\end{Claim}
\begin{proof}
The claim is essentially due to dynamical self-similarity and the SSC. Fix $\omega\in\Omega$, $u\in X_\infty^{(\omega)}$, and let $E\subset[-1,1]$ be a Borel subset. It follows from \eqref{eq:coding-map} that
\begin{equation} \label{eq:shift-ifs}
\Pi_\omega(u) = f_u^{(\omega_1)}(\Pi_{\sigma\omega}(\sigma u)).
\end{equation}
Write $x=\Pi_\omega(u)$, and note that for a Borel set $E\subset [-1,1]$ we have
$$S_{-\log(r_{\omega_1})}\eta_{\omega,u}(E)=c\cdot \eta^{(\omega)}(r_{\omega_1}\cdot E +x)=c\cdot f_{u}^{\omega_1}\eta^{(\sigma \omega)}(r_{\omega_{1}}\cdot E+x).$$ Indeed, by our assumption that the distance between the sets $f_v^{(\omega_1)}(Y^{(\sigma\omega)})$ is greater than $2$, at most one of them may intersect $r_{\omega_{1}}\cdot E+x$ (a set of diameter $<2$). Continuing this line of reasoning, we have
\begin{align*}
c\cdot f_{u}^{(\omega_1)}\eta^{(\sigma \omega)}(r_{\omega_{1}}\cdot E+x)&=c\cdot \eta^{(\sigma \omega)}\left(\frac{r_{\omega_{1}}\cdot E+x-t_{u}^{\omega_1}}{r_{\omega_1}}\right)\\
&=c\cdot \eta^{(\sigma \omega)}(E+(f_{u}^{(\omega_1)})^{-1}(x))\\
&=\eta_{M(\omega,u)},
\end{align*}
where we used \eqref{eq:shift-ifs} in the last line.
\end{proof}

\begin{Claim} \label{claim: Q gen.}
For $\mathbb{P}$-a.e. $\omega$, the measure $\eta^{(\omega)}$ generates the distribution $Q$ that is defined as follows: let $\overline{Q}$ be the suspension measure for the suspension flow with base $(\Omega',\mathbb{P}')$ and roof function $\rho(\omega,u)=-\log(r_{\omega_1})$. That is, denoting Lebesgue measure by $\lambda$,
\[
\overline{Q} = \frac{1}{\int -\log(r_{\omega_1}) \,d\mathbb{P}(\omega)} (\mathbb{P}'\times \lambda)|_{\{ (\omega,u,t): 0\le t< -\log(r_{\omega_1})  \}}.
\]
 Then $Q$ is the push-forward of $\overline{Q}$ under $(\omega,u,t)\mapsto S_t \eta_{\omega,u}$.
\end{Claim}
\begin{proof}
Our goal is to show that for $\mathbb{P}'$ almost all $(\omega,u)$, the measure $\eta^{(\omega)}$ generates $Q$ at $\Pi_{\omega}(u)$, or in other words, the scenery $(S_t \eta_{\omega,u})_{t=0}^\infty$ equidistributes for $Q$. Fix $f\in C(\mathcal{P}([-1,1]))$  and let
\[
F(\omega,u) := \int_0^{-\log r_{\omega_1}} f(S_t \eta_{\omega,u})\,dt.
\]
By  Claim \ref{Claim-M'},
\begin{equation} \label{eq:ergodic-sum-to-integral}
\sum_{j=0}^{N-1} F(M^j(\omega,u)) = \int_0^{\sum_{j=1}^{N}-\log(r_{\omega_j})} f(S_t\eta_{\omega,u})\, dt.
\end{equation}
By the ergodic theorem, for $\mathbb{P}'$-a.e. $(\omega,u)$,
$$\lim_{N\to\infty} \frac{1}{N}\sum_{j=0}^{N-1} F(M^j(\omega,u))=  \int_{\Omega}\int_{X_\infty^{(\omega)}}\int_0^{-\log r_{\omega_1}} f(S_t
\eta_{\omega,u}) \, dt\,d\overline{\eta}^{(\omega)}(u)\,d\mathbb{P}(\omega)$$ and
$$\lim_{N\to\infty} \frac{N}{\sum_{j=1}^{N} -\log(r_{\omega_j})} = \frac{1}{\sum_i \mathbb{P}([i]) \log(1/r_i)}.$$
Multiplying these two identities, recalling \eqref{eq:ergodic-sum-to-integral}, and using that $r_i$ is bounded away from $0$ and $1$, we see that for $\mathbb{P}$-almost all $\omega$ and $\overline{\eta}^{(\omega)}$-almost all $u$, we have
\[
\lim_{T\to\infty}\frac{1}{T}\int_0^T f(S_t \eta_{\omega,u})\,dt = \int f(\nu)\, dQ(\nu).
\]
Finally,  by the separability of $C(\mathcal{P}([-1,1]))$, for $\mathbb{P}$-almost all $\omega$ and $\overline{\eta}^{(\omega)}$-almost all $u,$ the equation above holds  for every $f\in C(\mathcal{P}([-1,1]))$. This proves the claim.
\end{proof}

\begin{proof}[Proof of Theorem \ref{Theorem model} in the orientation-preserving case]
First, by Claim \ref{claim: Q gen.}, for $\mathbb{P}$-a.e. $\omega$, the measure $\eta^{(\omega)}$ generates the distribution $Q$ as in Claim \ref{claim: Q gen.}. Furthermore, $Q$ is non-trivial, since the model is non-trivial. Finally, we need to establish the ergodicity of $Q$. This follows since $\overline{Q}$ is ergodic, being the suspension of an ergodic (in fact, Bernoulli) measure, and the map $(\omega,u,t)\mapsto S_t\eta_{\omega,u}$ taking $\overline{Q}$ to $Q$ is a factor map, as can be seen from   Claim \ref{Claim-M'}.
\end{proof}

We remark that our proof of Theorem \ref{Theorem model} in the orientation preserving case also works under the weaker assumption that $\Sigma$ is a non-trivial ergodic model. In this case $\mathbb{P}'$ is no longer Bernoulli but can still be seen to be ergodic. Our proof of this theorem in the orientation reversing case relies upon the ergodicity of certain group extensions. The ergodicity of these group extensions does not always hold if our model is just assumed to be ergodic. Fortunately however, ergodicity can be shown to hold if our model is Bernoulli.

\subsection{The case with some orientation-reversing map}
\label{subsec:orientation-reversing}

Throughout this section we will assume that our model is such that there exists at least one $i$ satisfying $r_{i}<0$. This case is slightly more complicated because as we zoom in on $\eta_{\omega,x},$ we are not necessarily going to see $\eta_{M(\omega,x)}$. Because of the presence of orientation reversing maps, it is possible that as we zoom in on $\eta_{\omega,x},$ we may see the image of $\eta_{M(\omega,x)}$ reflected around the origin. Our analogue of the map $M$ has to take this behaviour into account, and our distribution $Q$ also has to see these reflected measures.

We again assume the conditions of Theorem \ref{Theorem model} are satisfied and that the sets $(f_u^{(\omega_1)}(Y^{(\sigma\omega)}))_{u\in X_1^{(\omega)}}$ are at distance $>2$ from each other. We let $\mathbb{P}'$ be as in the orientation preserving case, and recall that its support is denoted by $\Omega'$. We define a map $M':\Omega'\times \mathbb{Z}_{2}\to \Omega'\times \mathbb{Z}_{2}$ as follows:
\[ M'(\omega,u,a) = \left\{ \begin{array}{ll}
\left(M(\omega,u),a\right) & \mbox{if $r_{\omega_1}>0$};\\
\left(M(\omega,u),a+1\right) & \mbox{if $r_{\omega_1}<0$}.\end{array} \right. \] We let $\mathfrak{m}$ denote the Haar measure on $\mathbb{Z}_{2}$ and let $\mathbb{P}''=\mathbb{P}'\times \mathfrak{m}.$

\begin{Lemma} \label{Lem-Markov-Chain}
$\mathbb{P}''$ is $M'$-invariant and ergodic; moreover, it is isomorphic to an ergodic Markov measure on an irreducible Markov chain.
\end{Lemma}
\begin{proof}
The state space of the Markov chain is $\mathcal{S}=\{ (i,u,a): i\in I, 1\le u\le k_i, a\in\mathbb{Z}_2\}$ and the probability of transitioning from $(i,u,a)$ to $(i',u',a')$ is $\mathbb{P}([i'])\cdot p_{u'}^{(i')}$ if $a=a'$ and $r_i>0$ or if $a\neq a'$ and $r_i<0$; and the transition is forbidden otherwise. Then
\[
\left( \frac{1}{2}\mathbb{P}([i])\cdot p_{u}^{(i)}\right)_{(i,u,a)\in\mathcal{S}}
\]
is readily checked to be a stationary probability for the Markov chain. The associated Markov measure $\widetilde{\mathbb{P}}''$ is supported on $\Omega'\times \mathbb{Z}_2^{\mathbb{N}}$ and satisfies
\[
\widetilde{\mathbb{P}}''([(\omega_1, u_1, a_1)\cdots (\omega_n, u_n, a_n)]) = \frac{1}{2} \mathbb{P}'([ (\omega_1, u_1)\cdots (\omega_n, u_n)]).
\]
Since $\omega$ and $a_1$ determine $a_j$ for $j\ge 2$ almost surely according to the transition rule, the projection $(\omega,u,a)\mapsto (\omega,u,a_1)$ provides a natural isomorphism between $(\Omega'\times \mathbb{Z}_2^{\mathbb{N}},\widetilde{\mathbb{P}}'',\sigma)$ and $(\Omega'\times\mathbb{Z}_2, \mathbb{P}'',M')$.

Finally, the subshift of finite type underlying the Markov chain is irreducible (in fact, one can get from one state to any other in at most two steps), so $\widetilde{\mathbb{P}}''$, and hence $\mathbb{P}''$, is ergodic.
\end{proof}

To each element of $\Omega'\times \mathbb{Z}_{2}$ we associate a measure as follows: Let $E\subset[-1,1]$ be a Borel set, then
\[ \eta_{\omega,u,a}(E) = \left\{ \begin{array}{ll}
\eta_{\omega,u}(E) & \mbox{if $a=0$};\\
\eta_{\omega,u}(-E) & \mbox{if $a=1$}.\end{array} \right. \]

\begin{Claim} \label{Claim-M''}
We have: $S_{-\log(|r_{\omega_1}|)}\eta_{\omega,u,a} = \eta_{M'(\omega,u,a)}$.
\end{Claim}
\begin{proof}
The claim follows from a case analysis based upon the sign of $r_{\omega_{1}}$ and the value of $a$. Note that the case where $r_{\omega_{1}}>0$ and $a=0$ is essentially  Claim \ref{Claim-M'}. For the purpose of brevity we do not cover each case here. We instead content ourselves with the case where $r_{\omega_{1}}<0$ and $a=0$. Let $E\subset [-1,1]$ be a Borel set. Then, using the separation between the sets   $(f_v^{(\omega_1)}(Y^{(\sigma\omega)}))_{v\in X_1^{(\omega)}}$, and denoting $x=\Pi_\omega(u)$, we have
\begin{align*}
S_{-\log(|r_{\omega_1}|)}\eta_{\omega,u,0}(E)&=c\cdot \eta^{(\omega)}(|r_{\omega_1}|\cdot E +x)\\
&=c \cdot f_{u}^{\omega_{1}}\eta^{(\sigma \omega)}(|r_{\omega_1}|\cdot E +x)\\
&= c \cdot \eta^{(\sigma \omega)}\left(\frac{|r_{\omega_1}|\cdot E +x-t_{u}^{\omega_1}}{r_{\omega_1}}\right)\\
&= c \cdot \eta^{(\sigma \omega)}\left(-E+(f_{u}^{\omega_{1}})^{-1}(x)\right)\\
&=\eta_{M(\omega,u)}(-E)=\eta_{M(\omega,u),1}(E)=\eta_{M'(\omega,u,0)}(E).
\end{align*}
\end{proof}

What remains of the proof of Theorem \ref{Theorem model} in the orientation reversing case is an almost identical argument to that presented in the orientation preserving case after Claim \ref{Claim-M'} is established.  For the purpose of completion, we mention that in this case the distribution $Q$ is defined as the push-forward of $\overline{Q}$ under the factor map $(\omega,x,a,t)\mapsto S_t\eta_{\omega,x,a}$, where $\overline{Q}$ is the suspension measure over $\mathbb{P}''$ with roof function $-\log |r_{\omega_1}|$.

\subsection{Removing the distance $>2$ assumption}
\label{subsec:removing-dist}

Finally, we discuss the general case when the smallest  distance between the sets $(f_u^{(\omega_1)}(Y^{(\sigma\omega)}))_{u\in X_1^{(\omega)}}$ is not larger than $2$ (note that by compactness, under the SSC there is indeed a smallest positive such distance). There are (at least) two ways to get around this. The first is to note that this can be achieved by rescaling all the translations of the IFS's in the model by a common factor. The second is, assuming $2 t_0>0$ is smaller than the smallest  distance between the sets $(f_u^{(\omega_1)}(Y^{(\sigma\omega)}))_{u\in X_1^{(\omega)}}$, to condition our measures not on $[-1,1]$ but rather on $[-t_0,t_0]$ in the magnification process. By the discussion in \cite[Section 1 and Section 3.1]{hochman2010dynamics}, the first way does not affect the uniformly scaling nature of the measures or the generated distributions, and the second has a corresponding very mild such effect.

\section{Proof of Theorem \ref{Main Theorem}} \label{Section proof of main result}

Let $S=(S_t)_{t=0}^\infty$ be a measurable (semi)flow on a space $X$, and let $P$ be $S$-invariant. Let $e(s)=\exp(2 \pi i s)$. The pure point spectrum $\Sigma(P,S)$ of $P$ is the set of all the $\alpha \in \mathbb{R}$ for which there exists  a non-zero measurable eigenfunction $\varphi:X \rightarrow \mathbb{C}$ such that $\varphi\circ S_t = e (\alpha t)\varphi$ for every $t\geq 0$, on a set of full $P$ measure. If $P$ is $S$-ergodic, then $|\varphi|$ is constant for any eigenfunction $\varphi$; in particular, any such measurable eigenfunction is in fact in $L^2(P)$. We also note that if $(Y,Q,S)$ is a factor of $(X,P,S')$ then $\Sigma(Q,S)\subset \Sigma(P,S')$: indeed, if $\Pi$ is the factor map, then any eigenfunction $\varphi$ on $Y$ gives rise to the eigenfunction $\varphi\circ\Pi$ (with the same eigenvalue) on $X$.

We now specialize to the scenery flow; recall Section \S\ref{Section outline}. In this case, the existence of an eigenvalue $\alpha$  indicates that some non-trivial feature of the measures generated by $P$ repeats periodically under magnification by $e^{1/\alpha}$. The following theorem of  Hochman and Shmerkin \cite[Theorems 1.1 and 1.2]{hochmanshmerkin2015} relates the pure point spectrum of a distribution generated by a measure with equidistribution. Recall that a distribution $P$ is called trivial if it is the distribution supported on $\delta_{0} \in \mathcal{M}^{\square}$.
\begin{theorem} \label{Thm: H-S}
Let $\mu$ be a measure generating a non-trivial $S$-ergodic distribution $P$, and let $\beta$ be a Pisot number. If $\Sigma(P,S)$ does not contain a nonzero integer multiple of $\frac{1}{\log \beta}$, then $\mu$ is pointwise $\beta$-normal. Furthermore, the same is true for $g\mu$ for all  $g \in \text{diff}^1 (\mathbb{R})$.
\end{theorem}

Finally, we recall a classical fact regarding eigenfunctions of suspensions over Markov measures; see \cite[Proposition 6.2]{ParryPollicott90} for the proof (of a more general statement).
\begin{theorem} \label{Thm:spectrum-suspension}
Let $(X,\mathbb{P},\sigma)$ be a subshift of finite type $X$ endowed with a Markov measure $\mathbb{P}$, let $\rho\in L^1(\mathbb{P})$ be a non-negative measurable roof function, and let $(X_\rho,\mathbb{P}_\rho,S)$ denote the corresponding suspension flow. If $\alpha\in \Sigma(\mathbb{P}_\rho,S)$, then there is $f\in C(X)$ such that
\begin{equation} \label{spectrum-suspension}
f(\sigma x) = e(\alpha \rho(x)) f(x), \quad x\in X.
\end{equation}
\end{theorem}

With these ingredients, we can conclude the proof of theorem \ref{Main Theorem}.

\begin{proof}[Proof of Theorem \ref{Main Theorem}]

Let $\mu$ be a self similar measure that satisfies the condition of Theorem \ref{Main Theorem} with respect to a Pisot number $\beta>1$, and let $g\in \text{diff}^1 (\mathbb{R})$. First, apply Theorem \ref{Theorem self similar} to obtain a disintegration of $\mu$ according to the corresponding model $\Sigma = ( \Phi^{(i)} _{i\in I}, (\mathbf{p}_i)_{i\in I}, \mathbb{P})$. That is,
\begin{equation*}
\mu = \int \eta^{(\omega)} d\mathbb{P}(\omega).
\end{equation*}
It follows that
\begin{equation*}
g\mu = \int g\eta^{(\omega)} d\mathbb{P}(\omega).
\end{equation*}
Therefore, to show that $g\mu$ is pointwise $\beta$-normal  it is enough to show that $\mathbb{P}$-almost surely the measure $g \eta^{(\omega)}$ is pointwise $\beta$-normal.

By Theorem \ref{Theorem self similar} there is some $j\in I$ such that $|r_j| \not \sim \beta$.  Also, by Theorem \ref{Theorem model}, for $\mathbb{P}$-a.e. $\omega$ the measure $\eta^{(\omega)}$  generates the non-trivial $S$-ergodic distribution $Q$, constructed in Section \S\ref{Section proof of Theorem model}. Moreover, $(\mathcal{M}^\square,Q,S)$ arises as a factor of a suspension flow with an ergodic Markov measure on the base - recall Claim \ref{claim: Q gen.}, Lemma \ref{Lem-Markov-Chain} and the discussion at the end of \S\ref{subsec:orientation-reversing}.

The state space of the corresponding Markov chain is either $\{ (i,u):i\in I, u\in \{1,\ldots, k_i\}\}$ in the orientation-preserving case (in this case the Markov measure is actually Bernoulli), and $\{ (i,u,a):i\in I, u\in \{1,\ldots, k_i\},a\in\mathbb{Z}_2\}$ when at least one IFS in the model is orientation-reversing. In both cases, the roof function is $-\log(|r_i|)$ where $i$ corresponds to the current state of the Markov chain.  Applying \eqref{spectrum-suspension} to the fixed point $(j,1)^\infty$ in the orientation-preserving case, or applying \eqref{spectrum-suspension} twice to the periodic point $((j,1,0)(j,1,1))^\infty$ otherwise (which is allowed since $f$ is continuous), we deduce that $2\alpha \log(|r_j|)\in\mathbb{Z}$.

Since $|r_j|\not\sim \beta$, we see that $\alpha$ cannot be a non-zero integer multiple of $1/\log(\beta)$. As we discussed at the beginning of this section, the same holds for the factor $(\mathcal{M}^\square,Q,S)$. Now the desired conclusion follows from Theorem \ref{Thm: H-S}.

\end{proof}

\bibliography{bib}{}
\bibliographystyle{plain}

\end{document}